\documentclass[12pt]{amsart}%
\usepackage{amsfonts}
\usepackage{amsmath}
\usepackage{amssymb}
\usepackage{graphicx}%
\newtheorem{theorem}{Theorem}
\newtheorem{lemma}{Lemma}
\newtheorem{corollary}{Corollary}

\theoremstyle{definition}
\newtheorem{definition}{Definition}

\theoremstyle{remark}

\numberwithin{equation}{section}

\begin{document}
\title[Point values and boundary values of analytic functions]{On distributional point values and boundary values of analytic functions}
\author[R. Estrada]{Ricardo Estrada}
\address{Mathematics Department\\
Louisiana State University\\
Baton Rouge, LA 70803, USA}
\email{restrada@math.lsu.edu}
\author[J. Vindas]{Jasson Vindas}
\address{Department of Mathematics, Ghent University, Krijgslaan 281 Gebouw S22, B 9000 Gent, Belgium}
\email{jvindas@cage.Ugent.be}
\begin{abstract}
We give the following version of Fatou's theorem for distributions that are boundary values of analytic functions. We
prove that if $f\in\mathcal{D}^{\prime}\left(  a,b\right)  $ is the
distributional limit of the analytic function $F$ defined in a region of the
form $\left(a,b\right)  \times\left(  0,R\right)  ,$ if $\ $the\ one sided
distributional limit exists, $f\left(  x_{0}+0\right)  =\gamma,$ and if $f$ is
distributionally bounded at $x=x_{0},$ then the \L ojasiewicz point value exists, $f\left(  x_{0}\right)  =\gamma$ distributionally, and in particular $F(z)\to \gamma$ as $z\to x_{0}$ in a non-tangential fashion.

\end{abstract}
\maketitle

\section{Introduction\label{Intro}}

The study of boundary values of analytic functions is an important subject in mathematics. In particular, it plays a vital role in the understanding of generalized functions \cite{beltrami,bremermann,c-k-p}. As well known, the behavior of an analytic function at the boundary points is intimately connected with the pointwise properties of the boundary generalized function \cite{EstradaComVar,estrada-vindasT2010,vindas-estradaT2008,vindas-estrada2008,vladimirov-d-z} and the study of this interplay has often an Abelian-Tauberian character. There is a vast literature on Abelian and Tauberian theorems for distributions (see the monographs \cite{estrada,ML,p-s-v,vladimirov-d-z} and references therein). 
 
In this article we present sufficient conditions for the existence of \L ojasiewicz point values \cite{lojasiewicz} for distributions that are boundary values of analytic functions. The pointwise notions for distributions used in this paper are explained in Section
\ref{Prelim}. The following
result by one of the authors is well known \cite{EstradaComVar}:

\textit{Suppose that }$f\in\mathcal{D}^{\prime}\left(  \mathbb{R}\right)
$\textit{ is the boundary value of a function }$F$, \textit{analytic in the
upper half-plane, that is, }$f\left(  x\right)  =F\left(  x+i0\right)
;$\textit{ if the distributional lateral limits }$f\left(  x_{0}\pm0\right)
=\gamma_{\pm}$ \textit{both exist, then }$\gamma_{+}=\gamma_{-}=\gamma
,$\textit{ and the distributional limit }$f\left(  x_{0}\right)
$\textit{\ exists and equals }$\gamma.$

On the other hand, the results of \cite{EstradaChina} imply that there are
distributions $f\left(  x\right)  =F\left(  x+i0\right)  $ for which one
distributional lateral limit exits but not the other. In Theorem \ref{TeoremaCaso2} we show that the existence of one the distributional lateral limits may be removed from the previous statement if an additional Tauberian-type  condition is assumed, namely, if the distribution is distributionally bounded at the point. We also show that when the distribution $f$ is a bounded function near the point, then the
distributional point value is of order 1. Furthermore, we give a general result of this kind for analytic functions that have distributional
limits on a contour.

As an immediate consequence of our results, we shall obtain the following version of Fatou's theorem \cite{koosis,Pom} for distributions that are boundary values of analytic functions.

\begin{corollary}
Let $F$ be analytic in a rectangular region of
the form $\left(  a,b\right)  \times\left(  0,R\right)  .$ Suppose that $f\left(
x\right)  =\lim_{y\rightarrow0^{+}}F\left(  x+iy\right)  $ in
$\mathcal{D}^{\prime}\left(  a,b\right)$, that $f$ is a bounded function near $x_{0}\in\left(a,b\right)$, and that the following average lateral limit exists
\begin{equation*}
\lim_{x\to x^{+}_{0}}\frac{1}{(x-x_{0})}\int_{x_{0}}^{x} f(t)\:\mathrm{d}t=\gamma\ .
\end{equation*}Then,
\begin{equation*}
\lim_{z\to x_{0}}F(z)=\gamma \ \ \ (\mbox{angularly}).
\end{equation*}
\end{corollary}

Finally, we remark that Theorem \ref{Thm:Cor} below generalizes some of our Tauberian results from \cite{vindas-estradaT2008}.

\section{Preliminaries\label{Prelim}}

We explain in this section several  pointwise notions for distributions. There are several equivalent ways to introduce them. We start with the useful approach from \cite{CF}. Define the operator $\mu_{a}$ on locally integrable complex valued functions in $\mathbb{R}$ as
\begin{equation*}
\mu_{a}\left\{  f\left(  t\right)  ;x\right\}  =\frac{1}{x-a}\int_{a}%
^{x}f\left(  t\right)  \,\mathrm{d}t\,,\ \ x\neq a\,, \label{4.1}%
\end{equation*}
while the operator $\partial_{a}$ is the inverse of $\mu_{a},$%
\begin{equation*}
\partial_{a}\left(  g\right)  =\left(  \left(  x-a\right)  g\left(  x\right)
\right)  ^{\prime}\,. \label{4.2}%
\end{equation*}
Suppose first that $f_{0}=f$ is \emph{real.} Then if it is bounded near $x=a,$
we can define
\begin{equation*}
\overline{f_{0}}\left(  a\right)  =\limsup_{x\rightarrow a}f\left(  x\right)
\,,\;\;\;\;\;\;\underline{f_{0}}\left(  a\right)  =\liminf_{x\rightarrow
a}f\left(  x\right)  \,. \label{4.3}%
\end{equation*}
Then $f_{1}=\mu_{a}\left(  f\right)  $ will be likewise bounded near $x=a$ and
actually
\begin{equation*}
\underline{f_{0}}\left(  a\right)  \leq\underline{f_{1}}\left(  a\right)
\leq\overline{f_{1}}\left(  a\right)  \leq\overline{f_{0}}\left(  a\right)
\end{equation*}
and, in particular, if $f\left(  a\right)  =f_{0}\left(  a\right)  $ exists,
then $f_{1}\left(  a\right)  $ also exists and $f_{1}\left(  a\right)
=f_{0}\left(  a\right)  .\smallskip$

\begin{definition}
A distribution $f\in\mathcal{D}^{\prime}\left(  \mathbb{R}\right)  $ is called
distributionally bounded at $x=a$ if there exist $n\in\mathbb{N}$ and
$f_{n}\in\mathcal{D}^{\prime}\left(  \mathbb{R}\right)  ,$ continuous and
bounded in a pointed neighborhood $\left(  a-\varepsilon,a\right)  \cup\left(
a,a+\varepsilon\right)  $ of $a,$ such that $f=\partial_{a}^{n}f_{n}%
.\smallskip$
\end{definition}

If $f_{0}$ is distributionally bounded at $x=a,$ then there exists a
\emph{unique} distributionally bounded distribution near $x=a,$ $f_{1},$ with
$f_{0}=\partial_{a}f_{1}.$ Therefore, $\partial_{a}$ and $\mu_{a}$ are
isomorphisms of the space of distributionally bounded distributions near
$x=a.$ Given $f_{0}$ we can form a sequence of distributionally bounded
distributions $\{f_{n}\}_{n=-\infty}^{\infty}$ with $f_{n}=\partial_{a}%
f_{n+1}$ for each $n\in\mathbb{Z}.$

We say that $f$ has the distributional point value $\gamma$ in the sense of
\L ojasiewicz \cite{lojasiewicz,estrada-vindasIntegral} and write
\begin{equation*}
f\left(  a\right)  =\gamma\;\;\;\left(  \mathrm{L}\right)  \,, \label{4.7}%
\end{equation*}
if there exists $n\in\mathbb{N},$ the order of the point value, such that
$f_{n}$ is continuous near $x=a$ and $f_{n}\left(  a\right)  =\gamma.$ 

It can
be shown \cite{CF, estrada, lojasiewicz, p-s-v} that $f\left(  a\right)  =\gamma$
$\left(  \mathrm{L}\right)  $ if and only if
\begin{equation*}
\lim_{\varepsilon\rightarrow0}f\left(  a+\varepsilon x\right)  =\gamma\,,
\end{equation*}
distributionally, that is, if and only if
\begin{equation}
\lim_{\varepsilon\rightarrow0^{+}}\left\langle f\left(  a+\varepsilon x\right)
,\phi\left(  x\right)  \right\rangle =\gamma\int_{-\infty}^{\infty}\phi\left(
x\right)  \,\mathrm{d}x\,, \label{4.9}%
\end{equation}
for each $\phi\in\mathcal{D}\left(  \mathbb{R}\right)  .$ On the other hand,
if $f$ is distributionally bounded at $x=a$ then $\left\langle f\left(  a+\varepsilon x\right)
,\phi\left(  x\right)  \right\rangle $ is bounded as $\varepsilon
\rightarrow0.$

We can also consider distributional lateral limits \cite{lojasiewicz,vindas2007}. We say that the
distributional lateral limit $f\left(  a+0\right)  $ $\left(  \mathrm{L}%
\right)  $ as $x\rightarrow a$ from the right exists and equals $\gamma,$ and
write%
\begin{equation*}
f\left(  a+0\right)  =\gamma\;\;\;\left(  \mathrm{L}\right)  \,, \label{4.10}%
\end{equation*}
if (\ref{4.9}) holds for all $\phi\in\mathcal{D}\left(  \mathbb{R}\right)  $
with support contained in $(0,\infty).$ The distributional lateral limit from
the left $f\left(  a-0\right)  $ $\left(  \mathrm{L}\right)  $ is defined in a
similar fashion.

Observe also that if $f=\partial_{a}^{n}f_{n},$ and $f_{n}$ is bounded near
$x=a,$ then $f\left(  a+0\right)  $ $\left(  \mathrm{L}\right)  $ exists, and
equals $\gamma$, if and only if $f_{n}\left(a+0\right)  =\gamma$ $\left(
\mathrm{L}\right).$

These notions have straightforward extensions to distributions defined in a
smooth contour of the complex plane. A natural extension of this pointwise notions for distributions is the so called quasiasymptotic behavior of distributions, explained, e.g., in \cite{p-s-v,vindas2010, vladimirov-d-z}.

\section{Boundary values \label{tauberianPV} and distributional point values}

We shall need the following well known fact \cite{beltrami}. We shall use the
notation $\mathbb{H}$ for the half plane $\left\{  z\in\mathbb{C}:\Im
m\,z>0\right\}.$

\begin{lemma}
Let $F$ be analytic in the half plane $\mathbb{H},$ and suppose that the
distributional limit $f\left(  x\right)  =F\left(  x+i0\right)  $ exists in
$\mathcal{D}^{\prime}\left(  \mathbb{R}\right)  .$ Suppose that there exists
an open, non-empty interval $I$ such that $f$ is equal to the constant
$\gamma$ in $I.$ Then $f=\gamma$ and $F=\gamma.\smallskip$
\end{lemma}

Actually using the theorem of Privalov \cite[Cor 6.14]{Pom} it is easy to see
that if $F$ is analytic in the half plane $\mathbb{H},$ $f\left(  x\right)
=F\left(  x+i0\right)  $ exists in $\mathcal{D}^{\prime}\left(  \mathbb{R}%
\right)  ,$ and there exists a subset $X\subset\mathbb{R}$ of non-zero measure
such that the distributional point value $f\left(  x_0\right)  $ exists and
equals $\gamma$ if $x_0\in X,$ then $f=\gamma$ and $F=\gamma.$

Our first result is for \emph{bounded} analytic
functions.

\begin{theorem}
\label{TeoremaCaso1}Let $F$ be analytic and bounded in a rectangular region of
the form $\left(  a,b\right)  \times\left(  0,R\right)  .$ Set $f\left(
x\right)  =\lim_{y\rightarrow0^{+}}F\left(  x+iy\right)  $ in
$\mathcal{D}^{\prime}\left(  a,b\right)$, so that $f\in L^{\infty}(a,b)$.  Let $x_{0}\in\left(  a,b\right)  $ be
such that
\begin{equation}
f\left( x_0+0\right)  =\gamma\;\;\;\left(  \mathrm{L}\right) 
\label{Ta.1}%
\end{equation}
exists. Then the distributional point value
\begin{equation}
f\left(  x_0\right)  =\gamma\;\;\;\left(  \mathrm{L}\right)
\label{Ta.2}%
\end{equation}
also exists. In fact, the point value is of the first order, and thus
\begin{equation}
\lim_{x\rightarrow x_{0}}\frac{1}{x-x_{0}}\int_{x_{0}}^{x}f\left(  t\right)
\,\mathrm{d}t=\gamma\,. \label{Ta.3}%
\end{equation}
\end{theorem}

\begin{proof}
We shall first show that it is enough to prove the result if the rectangular
region is the upper half-plane $\mathbb{H}.$ Indeed, let $\mathsf{C}$ be a smooth simple closed curve
contained in $\left(  a,b\right)  \times\lbrack0,R)$ such that $\mathsf{C}%
\cap\left(  a,b\right)  =[  x_{0}-\eta,x_{0}+\eta]  ,$ and which is
symmetric with respect to the line $\Re e\,z=x_{0}.$ Let $\varphi$ be a
conformal bijection from $\mathbb{H}$ to the region enclosed by $\mathsf{C}$
such that the image of the line $\Re e\,z=x_{0}$ is contained in $\Re
e\,z=x_{0},$ so that, in particular, $\varphi\left(  x_{0}\right)  =x_{0}.$
Then (\ref{Ta.1})--(\ref{Ta.3}) hold if and only if the
corresponding equations hold for
$f\circ\varphi$.

Therefore we may assume that $a=-\infty,$ and $b=R=\infty.$ In this case,
$f$ belongs to the Hardy space $H^{\infty},$ the closed subspace of $L^{\infty}\left(
\mathbb{R}\right)  $ consisting of the boundary values of bounded analytic
functions on $\mathbb{H}.$ Let $f_{\varepsilon}\left(  x\right)  =f\left(
x_{0}+\varepsilon x\right)  .$ Clearly, the set $\left\{  f_{\varepsilon
}:\varepsilon>0\right\}  $ is weak* bounded (as a subset of the dual space
$\left(  L^{1}\left(  \mathbb{R}\right)  \right)  ^{\prime}=L^{\infty}\left(
\mathbb{R}\right)  $) and, consequently, a relatively weak* compact set. If
$\left\{  \varepsilon_{n}\right\}  _{n=0}^{\infty}$ is a sequence of positive
numbers with $\varepsilon_{n}\rightarrow0$ such that the sequence $\left\{
f_{\varepsilon_{n}}\right\}  _{n=0}^{\infty}$ is weak* convergent to $g\in
L^{\infty}\left(  \mathbb{R}\right)  ,$ then $g\equiv\gamma,$ since $g\in
H^{\infty},$ and $g\left(  x\right)  =\gamma$ for $x>0.$ In fact, the condition (\ref{Ta.1}) means that 
$$
\int_{0}^{\infty}g(x)\psi(x) \mathrm{d}x = \lim_{n\to\infty} \int_{0}^{\infty}f_{\varepsilon_{n}}(x)\psi(x) \mathrm{d}x= \gamma \int_{0}^{\infty}\psi(x) \mathrm{d}x\ ,
$$
for all $\psi\in \mathcal{D}(0,\infty)$, which yields the claim. Since any sequence
$\left\{  f_{\varepsilon_{n}}\right\}  _{n=0}^{\infty}$ with $\varepsilon
_{n}\rightarrow0$ has a weak* convergent subsequence, and since that
subsequence converges to the constant function $\gamma,$ we conclude that
$f_{\varepsilon}\rightarrow\gamma$ in the weak* topology of $L^{\infty}\left(
\mathbb{R}\right)  .$ Furthermore, (\ref{Ta.3}) follows by taking $x=x_{0}+\varepsilon$ and
$\phi\left(  t\right)  =\chi_{\left[  0,1\right]  }\left(  t\right)  ,$ the
characteristic function of the unit interval, in the limit $\lim
_{\varepsilon\rightarrow0}\left\langle f_{\varepsilon}\left(  t\right)
,\phi\left(  t\right)  \right\rangle =\gamma\int_{-\infty}^{\infty}\phi\left(
t\right)  \,\mathrm{d}t.$ 
\end{proof}

We can now prove our main result, a distributional extension of Theorem \ref{TeoremaCaso1}.

\begin{theorem}
\label{TeoremaCaso2}Let $F$ be analytic in a rectangular region of the form
$\left(  a,b\right)  \times\left(  0,R\right)  .$ Suppose $f\left(  x\right)
=\lim_{y\rightarrow0^{+}}F\left(  x+iy\right)  $ in the space $\mathcal{D}%
^{\prime}\left(  a,b\right)  .$ Let $x_{0}\in\left(  a,b\right)  $ such that
$f\left(  x_{0}+0\right)  =\gamma$ $\left(  \mathrm{L}\right)  .$ If $f$ is
distributionally bounded at $x=x_{0}$ then $f\left(  x_{0}\right)  =\gamma$
$\left(  \mathrm{L}\right)  .$ Furthermore, $F(z)\to \gamma$ as $z\to x_{0}$ in an angular fashion.
\end{theorem}

\begin{proof}
There exists $n\in\mathbb{N}$ and a function $f_{n}$ bounded in a neighborhood
of $x_{0}$ such that $f=\partial_{x_{0}}^{n}f_{n};$ notice that $f\left(
x_{0}\right)  =\gamma$ $\left(  \mathrm{L}\right)  $ if and only if
$f_{n}\left(  x_{0}\right)  =\gamma$ $\left(  \mathrm{L}\right)  .$ But
$f_{n}\left(  x\right)  =F_{n}\left(  x+i0\right)  $ distributionally, where
$F_{n}$ is analytic in $\left(  a,b\right)  \times\left(  0,R\right)  ;$ here
$F_{n}$ is the only angularly bounded solution of $F\left(  z\right)
=\partial_{x_{0}}^{n}F_{n}\left(  z\right)  $ (derivatives with respect to
$z$). Clearly, $f_{n}(x)=F_{n}(x+i0)$. Since $f_{n}$ is bounded near $x=x_{0},$ $F_{n}$ is also bounded in a
rectangular region of the form $\left(  a_{1},b_{1}\right)  \times\left(
0,R_{1}\right)  ,$ where $x_{0}\in\left(  a_{1},b_{1}\right)  .$ Clearly
$f_{n}\left(  x_{0}+0\right)  =\gamma$ $\left(  \mathrm{L}\right),$ so the
Theorem \ref{TeoremaCaso1} yields $f_{n}\left(  x_{0}\right)  =\gamma$
$\left(  \mathrm{L}\right)  ,$ as required. Finally, the fact that $F(z)\to\gamma$ as $z\to x_{0}$, angularly,   
is a consequence of the existence of the distributional point value, as shown in \cite{estrada2005,vindas2010}.\end{proof}

Observe that in general the result (\ref{Ta.3}) does not follow if $f$ is not
bounded but just distributionally bounded near $x_{0}.$ 

We may use a conformal map to obtain the following general form of the Theorem
\ref{TeoremaCaso2}.

\begin{theorem}
\label{TeoremaCaso3}Let $\mathsf{C}$ be a smooth part of the boundary
$\partial\Omega$ of a region $\Omega$ of the complex plane. Let $F$ be
analytic in $\Omega,$ and suppose that $f\in\mathcal{D}^{\prime}\left(
\mathsf{C}\right)  $ is the distributional boundary limit of $F.$ Let $\xi
_{0}\in\mathsf{C}$ and suppose that the distributional lateral limit $f\left(
\xi_{0}+0\right)  =\gamma$ $\left(  \mathrm{L}\right)  $ exists and $f$ is
distributionally bounded at $\xi=\xi_{0},$ then $f\left(  \xi_{0}\right)
=\gamma$ $\left(  \mathrm{L}\right)  $ and $F(z)$ has non-tangential limit $\gamma$ at the boundary point $\xi_{0}$.
\end{theorem}

We also immediately obtain the following Tauberian theorem. As mentioned at the Introduction, it generalizes some Tauberian results by the authors from \cite{vindas-estradaT2008}.

\begin{theorem}
\label{Thm:Cor}Let $F$ be analytic in a rectangular region of the form
$\left(  a,b\right)  \times\left(  0,R\right)  .$ Suppose $f\left(  x\right)
=\lim_{y\rightarrow0^{+}}F\left(  x+iy\right)  $ in the space $\mathcal{D}%
^{\prime}\left(  a,b\right)  .$ Let $x_{0}\in\left(  a,b\right)  $ such that
the distributional limit $\lim_{y\rightarrow0^{+}}F\left(  x_{0}+iy\right)
=\gamma$ $\left(  \mathrm{L}\right)  $ exists. If $f$ is distributionally
bounded at $x=x_{0}$ then $f\left(  x_{0}\right)  =\gamma$ $\left(
\mathrm{L}\right)  $ and the angular (ordinary) limit exists: $\lim_{z\to x_{0}%
}F\left( z\right)  =\gamma.$
\end{theorem}

\begin{proof}
If we consider the curve $\mathsf{C}$ to be the union of the segments
$(a,x_{0}]$ and $[x_{0},iR),$ then the distributional lateral limit of the
boundary value of $F$ on $\mathsf{C}$ exists and equals $\gamma$ as we
approach $x_{0}$ from the right along $\mathsf{C}$ and so the Theorem
\ref{TeoremaCaso3}\ yields that the distributional limit from the left, which
is nothing but $f\left(  x_{0}-0\right)  $ $\left(  \text{L}\right)  ,$ also
exists and equals $\gamma.$ Then the Theorem \ref{TeoremaCaso2} gives us that
$f\left(  x_{0}\right)  =\gamma$ $\left(  \mathrm{L}\right)  .$ The existence
of the angular limit of $F\left(  z\right)  $ as $z\rightarrow x_{0}$ then
follows.
\end{proof}

\end{document}